\documentclass[12pt,reqno]{amsart}
\usepackage{tikz}
\usepackage{hyperref}
\usepackage{centernot}
\usepackage{ amssymb }

\newcommand \C{\mathcal{C}}
\newcommand \G{\mathcal{G}}

\newcommand \D{\mathcal{D}}

\newcommand{\ass}{\operatorname{Ass}}

\newtheorem{theorem}{Theorem}[section]
\newtheorem{definition}[theorem]{Definition}

\newtheorem{lemma}[theorem]{Lemma}

\newtheorem{remark}[theorem]{Remark}

\newtheorem{corollary}[theorem]{Corollary}

\newtheorem*{notation*}{Notation}

\begin{document}

\title{ordinary and symbolic powers of edge ideals of weighted oriented graphs}

\author[A.~Banerjee]{Arindam Banerjee}
\email{123.arindam@gmail.com}
\address{Ramakrishna Mission Vivekenanda Educational and Research Institute, Belur, West
	Bengal, India}

\author[K. K. Das]{Kanoy Kumar Das}
\email{kanoydas0296@gmail.com}
\address{Ramakrishna Mission Vivekananda Educational and Research Institute, Belur, West
	Bengal, India}

\author[S. Selvaraja]{S. Selvaraja}
\email{selva.y2s@gmail.com, sselvaraja@cmi.ac.in}
\address{Chennai Mathematical Institute, H1, SIPCOT IT Park, Siruseri, Kelambakkam, Chennai 603103, Tamil Nadu, India}

\maketitle

\begin{abstract}
  Let $\D$ be a weighted oriented graph and $I(\D)$ be its edge ideal. 
  In this paper, we 
  show that all the symbolic and ordinary powers of 
  $I(\D)$ coincide when $\D$ is a 
  weighted oriented certain class of tree.
  Finally, we give  necessary and sufficient conditions for the equality of 
  ordinary and symbolic powers of naturally oriented lines. 
\end{abstract}
\section{Introduction}

Symbolic powers of ideals have been studied intensely over the last two decades.
We refer the
reader to \cite{DDAGHN} for a review of results in the literature.
Let $R = k[x_1,\ldots,x_n]$ be the polynomial ring in $n$ variables over a field $k$ 
and $I$ be a homogeneous ideal in $R$. 
Then the $s$-th symbolic powers of $I$ is defined by 
$I^{(s)}=\bigcap\limits_{\mathfrak{p} \in \ass(I)}I^sR_{\mathfrak{p}} \cap R$, where 
$\ass(I)$ is the set of  associated prime ideals of $I$. 
Comparing the ordinary and symbolic powers of an ideal becomes a natural question. 
In general, the question of when the symbolic and ordinary powers of a given ideal coincide
is open. There are conditions on $I$ that are equivalent to $I^{(s)} = I^s$ 
for all $s \geq 1$ 
given by Hochster \cite{H73} when $I$ is prime, and 
generalized by Li and Swanson \cite{LS06} to the case
when $I$ is a radical ideal. Other sufficient conditions are known, and the question is 
completely settled in certain cases. However, the above question is still open even for 
monomial ideals. In this paper we study the 
ordinary and symbolic powers of edge ideals of weighted oriented graphs.

Let $G = (V(G), E(G))$ be a finite simple 
(no loops, no multiple edges) undirected graph. A weighted oriented graph $\D$ whose underlying graph is $G$, 
is a triplet $(V(\D), E(\D), w)$ where $V(\D) = V(G)$, $E(\D) \subseteq V(\D) \times V(\D)$ such that 
$\{\{x, y\} \mid \{x, y\} \in E(\D)\} = E(G)$ and $w$ is a function $w : V(\D) \longrightarrow \mathbb{N}$. 
The weight of $x_i \in V(\D)$ is $w(x_i)$ or $w_{i}$. 
An ordered pair $\{x,y\}\in E(\D)$ if there is a directed edge from the vertex $x$ to $y$.
Let $\D = (V(\D) , E(\D), w)$ denote a weighted oriented graph with vertices
$V(\D) = \{x_1,\ldots,x_n\}$. By identifying the vertices with the variables in the polynomial ring 
$R=k[x_1,\ldots, x_n],$  we can associate to
each weighted oriented graph $\D$ a monomial ideal $I(\D)$ generated by the set 
$\{x_i x_j^{w(x_j)} \mid \{x_i , x_j \}  \in E (\D) \}$.
The ideal $I(\D)$ is called the \textit{edge ideal} of $\D$. 
These ideals generalize the usual edge ideals of graphs, since if
$w(x)=1$ for all $x \in V(\D)$, then $I(\D)=I(G)$.
The interest in studying edge ideals of weighted oriented graphs has its foundation in coding theory, 
specifically in the study of Reed-Muller type codes (see \cite{ReedMuller}).
These edge ideals of weighted oriented graphs appears as initial ideals of vanishing ideals of 
projective spaces over finite fields.
Since then, the researchers have been investigating the connection
between the combinatorial properties of the weighted oriented graphs and the algebraic
properties of the corresponding edge ideals (\cite{HLMRV19}, \cite{ReedMuller}).

In \cite{BKS}, the authors proved that if $I(\D)$ is the edge ideal
of a weighted oriented graph $\D$, then $I(\D)$ has a linear resolution if and only if 
all powers
of $I(\D)$ have linear resolution.
As an immediate consequence, we have if $I(\D)$ has a linear resolution, then
$I(\D)^{(s)}=I(\D)^s$ for all $s \geq 1$.
In \cite{BBKMS}, the authors 
proved that the equality of ordinary and symbolic powers of edge
ideals of weighted oriented graphs when the underlying simple graph is either a complete
graph or a complete bipartite graph.
Mandal and Pradhan studied the symbolic powers of 
edge ideals of weighted oriented graphs.
Recently, they showed that 
if $\D$ is any weighted oriented star graph or some specific weighted naturally 
oriented path, then $I(\D)^s=I(\D)^{(s)}$ for all $s \geq 2$,
(\cite{MandalPradhan}, \cite{2105.11429}). In this paper, we study the 
ordinary and symbolic powers of edge ideals of some classes of oriented tree.
Recall that a \textit{tree} is a graph in which there exists a unique path between every pair of 
distinct vertices. A \textit{rooted tree} is a tree together with a fixed 
vertex called the root. In particular, 
in a rooted tree there exists a unique path from the root to any given vertex. 
More precisely, we prove the following. Let $\mathcal D$ be a weighted oriented graph with $w(v)>1$ for all 
$v\in V(\D)$. 
If $I(\D)=(xy^{w(y)}, yz^{w(z)}) + I(\D')$ where $\D'$ is an oriented rooted tree with
root at the vertex $z$, then $I(\D)^{(s)}=I(\D)^s$ for all $s\geq 2$ (Theorem \ref{forest}).
Finally, we completely characterize weighted oriented uni-directed lines for 
which all the ordinary and symbolic powers coincide. More precisely, 
let $\mathcal D$ be a weighted oriented line with 
$V(\D)=\{x_1, x_2, x_3, \ldots , x_n\}$ and 
$E(\D)=\{(x_i,x_{i+1}) \mid 1\leq i\leq n-1\}$. 
Then $I(\D)^{(s)}=I(\D)^s$ for all $s\geq 1$ if and only if the following condition is 
satisfied : if $w_j\geq 2$ for some $1\leq j\leq n$ then $w_i\geq 2$ for all 
$j\leq i\leq n-1$ (Theorem \ref{main line}).

\section{Preliminaries}\label{preliminaries}

In this section, we set up the basic definitions and notation needed for the main results.
If $x$ is a vertex of $\D$, we define $N_{\D}^+(x)=\{y\in \D \mid \{x,y\}\in E(\D)\}$ and 
$N_{\D}^-(x)=\{y\in \D \mid \{y,x\}\in E(\D)\}$ to be the $out$-$neighbourhood$ and the 
$in$-$neighbourhood$ of $x$. Then the neighbourhood of $x$ is the set 
$N_{\D}(x)=N_{\D}^+(x)\cup N_{\D}^-(x)$. A non-isolated vertex $x\in V(\D)$ is said to be source 
(respectively, sink) if it does not have any in-neighbourhood (respectively, out-neighbourhood).
	
	We now recall some useful notions about weighted oriented graphs as 
	developed in \cite{PitonesReyesToledo} which is key to this paper.
\begin{definition}\cite[Definition 4]{PitonesReyesToledo}
Let $\D$ be an weighted oriented graph, and let $\G$ be its underlying graph. Given a vertex cover $C$ of $\G$, we define 
\begin{align*}
 L_1(C)&=\{x\in C | N_{\D}^+(x)\cap C^c\neq \phi\},\\
 L_2(C)&=\{x\in C |x\notin L_1(C) \text{ and } N_{\D}^-(x)\cap C^c\neq \phi\},\\
L_3(C)&=C\setminus(L_1(C)\cup L_2(C)),
\end{align*}
where $C^c$ is the complement of $C$.
\end{definition}
	
	\begin{definition}\cite[Definition 7]{PitonesReyesToledo}
		A vertex cover $C$ of $\D$ is said to be strong if for each $x\in L_3(C)$ 
		there is $(y,x)\in E(D)$ such that $y\in L_2(C)\cup L_3(C)$ and 
		$w(y)\geq 2$.
	\end{definition}
	\begin{remark}
		Any minimal vertex cover of $\D$ is a strong vertex cover.
	\end{remark}

	\begin{definition}\cite[Definition 19]{PitonesReyesToledo}
		Let $C$ be a strong vertex cover of $\D$. The irreducible ideal 
		associated to $C$ is the ideal 
		$$I_C:=(L_1(C)\cup \{x_j^{w_j}|x_j\in L_1(C)\cup L_2(C)\})$$
	\end{definition}
	\begin{remark}\cite[Lemma 20]{PitonesReyesToledo}
		$I(\D)\subseteq I_C$ for all strong vertex cover $C$ of $\D$.
	\end{remark}
	
	\begin{theorem}\cite[Theorem 25]{PitonesReyesToledo}\label{irrdec}
		Let $\mathcal C_S(\D)$ be the set of all strong vertex covers of $\D$. 
		Then the irredundant irreducible decomposition of 
		$I(\D)$ is given by $$I(\D)=\bigcap_{C\in \mathcal{C}_S(\D)}I_C.$$ In particular, there is an one to one correspondence between the associated primes of $I(\D)$ and the strong vertex covers of $\D$.
	\end{theorem}

Let $I=Q_1\cap \cdots \cap Q_m$ be a primary 
decomposition of the ideal $I$. For $P\in \ass(R/I)$, we denote $Q_{\subseteq P}$ to be the intersection of all $Q_i$ with $\sqrt{Q_i}\subseteq P$. 
\begin{theorem}\cite[Theorem 3.7]{SymbolicMon}\label{altsym}
	The $m$-th symbolic power of a monomial ideal $I$ is $$I^{(m)}=\bigcap_{P\in
	\ass(R/I)}Q^m_{\subseteq P}$$
\end{theorem}

\section{Symbolic Powers of Weighted Oriented Graphs}\label{symb power}
In this section, we will study equality of ordinary and symbolic powers of edge ideals of 
certain classes of weighted oriented graphs. 

The following lemma exhibits a situation when all the symbolic and ordinary powers of the edge ideals coincide.
	\begin{lemma}
		Let $\D$ be a weighted oriented graph such that $w(v)>1$ for all
		$v\in V(\D)$. Then $V(\D)$ is a strong vertex cover of $\D$ if and only if 
		$\D$ does not contain any source vertex. Moreover, 
		if $V(\D)$ is a strong vertex cover, then $I(\D)^{(s)}=I(\D)^s$ for all 
		$s \geq 1$.
	\end{lemma}
	\begin{proof}
	From the definition of strong vertex cover, $V(\D)$ is a strong vertex cover of 
	$\D$ if and only if for all $x\in V(\D)$ there is $y\in V(\D)$ such that 
	$\{y,x\}\in E(\D)$. In other words, $V(\D)$ is a strong vertex cover if and only
	if $\D $ does not contain any source vetices. Now, if $V(\D)$ is a strong vertex cover,
	then the unique homogeneous maximal ideal is an associated prime of $I(\D)$, and 
	hence by \cite[Lemma 3.3]{SymbolicMon}, $I(\D)^{(s)}=I(\D)^s$ for all $s \geq 1$.
	\end{proof}
As an immediate consequence, we have the following: 
\begin{corollary}\label{cycle}
	Let $\D$ be a weighted naturally oriented cycle such that
	$w(v)> 1$ for all $v\in V(\D)$. Then $I(\D)^{(s)}=I(\D)^s$ for all $s\geq 1$.
\end{corollary}
	\begin{theorem}\label{forest}
		Let $\mathcal D$ be a weighted oriented graph with 
		$w(v)>1$ for all $v\in V(\D)$. If $I(\D)=(xy^{w(y)}, yz^{w(z)}) + I(\D')$,
		where $\D'$ is an oriented rooted tree with root at the vertex $z$, then 
		$I(\D)^{(s)}=I(\D)^s$ for all $s\geq 1$.
	\end{theorem}

	\begin{proof}
		We first observe that if the 
		strong vertex cover of 
		$\mathcal D$ contains the vertex $x$, then it can not contain the vertex $y$, and 
		hence it must contain $z$. 
		In view of this, we partition the strong vertex covers of $\D$ 
		into two maximal sub-classes, namely $\mathcal{C}_1$ and $\mathcal{C}_2$,
		where $\mathcal{C}_1$ contains all the strong vertex cover of 
		$\mathcal{D}$ which contains the vertices $\{x, z\}$ and does not contain the vertex $y$ 
		and $\mathcal{C}_2$ contains all the strong vertex cover of $\mathcal{D}$ 
		which contains the vertex $y$ and does not contain the vertex $x$. 
		Set $\mathcal{C}=\mathcal{C}_1\cup \mathcal{C}_2$. 
		By Theorem \ref{irrdec}, the irredundant primary decomposition of 
		$I$ is given by $$I=I(\D)=\bigcap_{C\in \mathcal{C}}I_C,$$ 
		where $I_C$'s are the irreducible ideal associated to the strong vertex cover 
		$C$. Note that $\{x,z\}\cup V(\D')$ is the unique maximal strong vertex cover 
		in $\C_1$ and $\{y,z\}\cup V(\D')$ is the unique maximal strong vertex cover 
		in $\C_2$. We can write 
		$$I=\bigcap_{C\in \mathcal{C}}I_C=\left(\bigcap\limits_{C\in \mathcal{C}_1}I_C\right) \bigcap 
		\left(\bigcap\limits_{C'\in \mathcal{C}_2}I_{C'}\right)=I_1\cap I_2,$$ where 
		$I_1=\bigcap\limits_{C\in \mathcal{C}_1}I_C$ and 
		$I_2=\bigcap\limits_{C'\in \mathcal{C}_2}I_{C'}$. 
		Then by Theorem \ref{altsym},  for any $s\geq 1$, $I^{(s)}=I_1^s\cap I_2^s$.	 
		\vskip 1mm \noindent
		\textbf{Claim:}
		$I_1=(x,z^{w(z)})+I(\D')$ and 
		$I_2=(y^{w(y)},yz^{w(z)})+I(\D')$.
		\vskip 1mm 
		\noindent
		\textit{Proof of the claim:} For any $C\in \mathcal{C}_1$, we have 
		$x\in C$ but $y\notin C$, therefore $x\in L_1(C)$ and $z\in L_2(C)$ and 
		hence $x\in I_C$ and $z^{w(z)}\in I_C$. 
		Also $I(\D)\subseteq I_C$ for any strong vertex cover $C$ of $\D$ and $\D'$ being a subgraph of $\D$ we have $I(\D')\subseteq I_1$. 
		Conversely, let $m=v_1^{r_1}v_2^{r_2}\cdots v_n^{r_n}$ be a minimal generator of $I_1$ 
		such that $m\notin I_1$.
		Then we must have $v_i\neq x$ for all $i$. Now, if for all $i>1$,
		$r_i<w(v)$ for any vertex $v\in V(\D)$, then if we consider the strong vertex cover $C=\{x, z\}\cup V(\D')\in \C_1$, 
		we have $m\notin I_C$, which is a contradiction. Otherwise if for some $1\leq i_1<i_2<\cdots < i_k\leq n$, 
		$r_{i_j}\geq w(v_{i_j})$ and $r_i<w(v_i)$ for all $i\neq i_j$, with $v_{i_j}\neq z$ for all $j$, then, by our assumption, for all $v_l\in N_{\D}^-(v_{i_j})$, $v_l\centernot\mid m$. Now let $S=\{v_{i_1},\ldots , v_{i_k}\}$ and  $T=N_{\D}^-(S)$. Observe that $C^*=C-(N_{\D}^+(T)\cup \{x,z\})$ is a strong vertex cover of $\D$ and $C^*\in \C_1$. Note that, $v_{i_j}\notin C^*$ for all $1\leq j\leq k$; $L_1(C^*)=T$ and hence for any $v_l, l\neq i_j, 1\leq j\leq k$ such that $v_l|m$, we have $v_l\notin L_1(C^*)$. Then we must have $m\notin I_{C^*}$, which is again a contradiction.
		Now proceeding as in the above, one can show that 
		$I_2=(y^{w(y)},yz^{w(z)})+I(\D')$.

		To show $I^{(s)}=I_1^s\cap I_2^s=I^s$ for all $s \geq 1$, it is enough to show 
		$$(x,z^{w(z)})^s\cap (y^{w(y)},yz^{w(z)})^s\subseteq 
		(xy^{w(y)},yz^{w(z)})^s \text{ for all $s \geq 1$.}$$  
		Let $m\in (x,z^{w(z)})^s\cap (y^{w(y)},yz^{w(z)})^s$. Then
		$x^ay^{bw(z)}|m$ and $y^{cw(y)}(yz^{w(z)})^d|m$ for some 
		$a,b,c,d\geq 0$ and $a+b=s,~ c+d=s$. Therefore 
		$x^ay^{cw(y)+d}z^{\max\{b,d\}w(z)}|m$. 
		Let $m'=x^ay^{cw(y)+d}z^{\max\{b,d\}w(z)}$. If $d\geq b$, then 
		we have $a\geq c$, so we can write $m'= x^{c-a}(xy^{w(y)})^c\cdot (yz^{w(z)})^d$ 
		and we are done. If $d<b$, then $a<c$. We can  
		write $m'=(xy^{w(y)})^ay^{(c-a)w(y)+d}\cdot z^{bw(z)}=
		(xy^{w(y)})^ay^{(b-d)w(y)+d}z^{bw(z)}$. Since $w(y)\geq 2$, we have $(b-d)w(y)+d\geq 2(b-d)+d=2b-d> b$ as
		$b>d$. Then we have $m'\in (xy^{w(y)},yz^{w(z)})^s$, and this 
		completes the proof.				
	\end{proof}
	
	\begin{remark}
		If $v$ is a sink vertex for some $v\in V(\D)$, if we assume $w(v)=1$ then the strong vertex covers of $\mathcal{D}$ remains the same as the case when $w(v)\geq 2$. So, we can just replace $w(v)$ by 1 everywhere in the above Theorem to get $I(\D)^{(s)}=I(\D)^s$ for all s.   
	\end{remark}
The next lemma is crucial in proving our main result.

\begin{lemma}\label{3rdsym}
	Let $\mathcal D$ be a weighted oriented uni-directed line with 
	$V(\D)=\{x_1, x_2, x_3, \ldots , x_n\}$ and $E(\D)=\{(x_i,x_{i+1})|1\leq i\leq n-1\}$.
	If for some $1< i < n-1$, $w(x_i)\geq 2, w(x_{i+1})=1$, then 
	$I(\D)^{(3)}\neq I(\D)^3$.	
\end{lemma}
\begin{proof}
	First we claim that $f=x_{i-1}x_i^{w_i}x_{i+1}^2x_{i+2}^{w_{i+2}}\in I(\D)^{(3)}$. 
	Note that for any maximal strong vertex cover $C$ of $\D$, $x_{i+2}\notin L_3(C)$. 
	Suppose $C$ contains the vertex $x_{i+2}$. Then either $x_{i+1}\notin C$ or 
	$x_{i+3}\notin C$. If $x_{i+1}\notin C$, then for any strong vertex cover 
	$C'\subseteq C$ 
	we have $x_i\in L_1(C')$ and $x_{i+2}\in L_1(C')\cup L_2(C')$. Then we can write 
	$f=(x_i)^2\cdot x_{i+2}^{w_{i+2}}\cdot x_{i-1}x_i^{w_i-2}x_{i+1}^2\in (I_{\subseteq C})^3$.
	If  $x_{i+3}\notin C$, then for any strong vertex cover $C'\subseteq C$ 
	we have $x_{i+2}\in L_1(C')$. Then we can write 
	$f=(x_ix_{i+1})^2\cdot x_{i+2}\cdot x_{i-1}x_{i+2}^{w_{i+2}-1}\in (I_{\subseteq C})^3$.
	Next, if $C$ does not contain the vertex $x_{i+2}$, then for any strong vertex cover 
	$C'\subseteq C$ we have $x_{i+1}\in L_1(C')$. Then we can write 
	$f=(x_{i-1}x_i^{w_i})\cdot (x_{i+1})^2\cdot x_{i+2}^{w_{i+2}} \in (I_{\subseteq C})^3 $. 
	Therefore, for any maximal strong vertex cover $C$, 
	we have $f\in (I_{\subseteq C})^3$ and hence $f\in I(\D)^{(3)}$. 
	Clearly, $f\notin I(\D)^3$. Therefore $I(\D)^{(3)}\neq I(\D)^3$. 
\end{proof}
Finally we present our main theorem of this section.
	\begin{theorem}\label{main line}
	Let $\mathcal D$ be a weighted oriented uni-directed line with $V(\D)=\{x_1, x_2, x_3, \ldots , x_n\}$ and $E(\D)=\{(x_i,x_{i+1})|1\leq i\leq n-1\}$. Then $I(\D)^{(s)}=I(\D)^s$ for all $s\geq 1$ if and only if the following condition is satisfied : if $w_j\geq 2$ for some $1< j< n$ then $w_i\geq 2$ for all $j\leq i\leq n-1$.
	\end{theorem}
\begin{proof}
	If for some $1< j < n-1$, $w(x_j)\geq 2, w(x_{j+1})=1$, then by Lemma \ref{3rdsym}, $I(\D)^{(3)}\neq I(\D)^3$.
	Conversely, let $1< k< n$ be such that $w(x_i)=1$ for all $i<k$ and $w(x_i)\geq 2$ 
	for all $k\leq i\leq n-1$. Now we will determine the maximal strong vertex 
	covers of $\D$.
	
	Let $\{C_{\alpha}\}$ be the collection of all maximal strong 
	vertex covers which does not contain the vertex $x_k$. 
	Then we have $x_{k-1},x_{k+1}\in C_\alpha$.
	Therefore,
	\begin{align*}
	 C_{\alpha}&=\Big\{C'_\alpha \mid C'_\alpha 
	\text{ is a minimal vertex cover of the induced line } \D(x_1,\ldots ,x_{k-2})\Big\}\\
	&\bigcup \{x_{k-1},x_{k+1},x_{k+2},x_{k+3}, \ldots ,x_n\}.
	\end{align*}
        
        Let $C$ be a strong vertex cover which contains the vertex $x_k$. 
	Note that, $x_k\notin L_3(C)$.
	Let $\{C_{\beta}\}$ be the collection of all maximal strong vertex covers which  
	contain the vertices $x_k$, $x_{k+1}$ and does not contain the vertex $x_{k-1}$. 
	Therefore,
	\begin{align*}
	 C_{\beta}&=\Big\{C'_\beta \mid C'_\beta \text{ is a minimal vertex cover of the
	induced line } \D(x_1,\ldots ,x_{k-3})\Big\}\\
	&\bigcup \{x_{k-2},x_k,x_{k+1},x_{k+2},
	\ldots ,x_n\}.
	\end{align*}

	Let $\{C_{\gamma}\}$ be the collection of all maximal strong vertex cover which  
	contain the vertices $x_k$, $x_{k-1}$ and does not contain the vertex $x_{k+1}$. 
	Therefore,
	\begin{align*}
	 C_{\gamma}&=\Big\{C'_\gamma \mid C'_\gamma 
	\text{ is a minimal vertex cover of the induced line } \D(x_1,\ldots ,x_{k-4})\Big\}\\
	&\bigcup \{x_{k-3},x_{k-1},x_k,\newline x_{k+2},x_{k+3},\ldots ,x_n\}.
	\end{align*}

	For a fixed maximal strong vertex cover $C_{\lambda} $ of $\D$, let 
	$I_{\subseteq C_{\lambda}}$ denotes the intersection 
	$$I_{\subseteq C_{\lambda}}=\bigcap_{C\subseteq C_{\lambda}}I_C,$$ where $C\subseteq C_{\lambda}$ is a 
	strong vertex cover. 
	Then we have, 
	$$I^{(s)}=I(\D)^{(s)}=\bigcap_{\lambda}I_{\subseteq C_{\lambda}}^s = \bigcap_\alpha 
	(I_{\subseteq C_{\alpha}})^s  \bigcap_\beta (I_{\subseteq 
	C_{\beta}})^s\bigcap_\gamma (I_{\subseteq C_{\gamma}})^s.$$
	
	By similar arguments, as in the claim of Theorem \ref{forest}, 
	we have 
	\begin{align*}
	 I_{\subseteq C_{\alpha}}&=(C'_\alpha)+ J_1,
	I_{\subseteq C_{\beta}}=(C'_\beta)+ J_2 \text{ and }
	I_{\subseteq C_{\gamma}}=(C'_\gamma) + J_3,
	\end{align*}
 where 
 \begin{align*}
  J_1&=\left(\{x_{k-1}, x_{k+1}^{w_{k+1}}, x_{k+1}x_{k+2}^{w_{k+2}},\ldots,
	x_{n-1}x_n^{w_n}\}\right),\\
 J_2&=\left(\{x_{k-2},x_k^{w_k},x_kx_{k+1}^{w_{k+1}}, 
	x_{k+1}x_{k+2}^{w_{k+2}},\ldots ,x_{n-1}x_n^{w_n}\}\right),\\
J_3&=\left(\{x_{k-3}, x_{k-1},x_k,x_{k+2}^{w_{k+2}},x_{k+2}x_{k+3}^{w_{k+3}}, 
	\newline \ldots ,x_{n-1}x_n^{w_n}\}\right).	
 \end{align*}
	
	Then we have, $I_{\subseteq C_{\alpha}}=I_{C'_\alpha}+J_1$, $I_{\subseteq C_{\beta}}=I_{C'_\beta}+J_2$ and $I_{\subseteq C_{\gamma}}=I_{C'_\gamma}+J_3$.
	Now, $\bigcap\limits_\alpha (I_{\subseteq C_{\alpha}})^s=\bigcap\limits_\alpha 
	(I_{C'_\alpha} +J_1)^s=\sum\limits_{s_1+s_2=s}\bigcap\limits_\alpha 
	(I_{C'_\alpha}^{s_1})\cdot J_1^{s_2}=\sum\limits_{s_1+s_2=s}
	I(\D(x_1,\ldots ,x_{k-2}))^{(s_1)}\cdot J_1^{s_2}=\sum\limits_{s_1+s_2=s}
	I(\D(x_1,\ldots ,x_{k-2}))^{s_1}\cdot J_1^{s_2}$. 
	And similarly we get, 
	$\cap_\beta (I_{\subseteq C_{\beta}})^s =\sum\limits_{s_1+s_2=s}
	I(\D(x_1,\ldots ,x_{k-3}))^{s_1}\cdot J_2^{s_2}$ and 
	$\cap_\gamma (I_{\subseteq C_{\gamma}})^s=\sum\limits_{s_1+s_2=s}
	I(\D(x_1,\ldots ,x_{k-4}))^{s_1}\cdot J_3^{s_2}$.
	
	To show $\bigcap\limits_\alpha (I_{\subseteq C_{\alpha}})^s\bigcap\limits_\beta 
	(I_{\subseteq C_{\beta}})^s\bigcap\limits_\gamma (I_{\subseteq 
	C_{\gamma}})^s\subseteq I^s$ for all $s\geq 2$, it is enough to show 
	$J_1^s\cap J_2^s\cap J_3^s \subseteq I^s$ for all $s\geq 2$.
	This reduces to showing the fact that, $$(x_{k-1}, x_{k+1}^{w_{k+1}})^s\cap (x_{k-2},
	x_{k}^{w_{k}})^s\cap (x_{k-3}, x_{k-1},x_k,x_{k+2}^{w_{k+2}})^s\subseteq I^s \text{ for all $s\geq 1$.} $$ 
	 
	 We prove this by induction on $s$. 
	The case $s=1$ is true. 
	Note that any element of $(x_{k-1}, x_{k+1}^{w_{k+1}})^s\cap (x_{k-2}, 
	x_{k}^{w_{k}})^s\cap (x_{k-3}, x_{k-1},x_k,x_{k+2}^{w_{k+2}})^s$ is divisible by the 
	monomial $f=x_{k-3}^ax_{k-2}^mx_{k-1}^{\max\{i,b\}}x_k^{\max\{nw_{k},c\}}
	x_{k+1}^{jw_{k+1}}
	x_{k+2}^{dw_{k+2}}$, for some $i,j,m,n,a,b,c,$ $d \geq 0$ with  
	$i+j=s$, $m+n=s$ and $a+b+c+d=s$. 
	Suppose $\max\{i,b\}>0$. Then we have $x_{k-1}|f$.  
	If $m\neq 0$, then we also have $x_{k-2}|f$. Then divide the monomial $f$ by the 
	edge $x_{k-2}x_{k-1}$ to get $$f'=x_{k-3}^ax_{k-2}^{m-1}x_{k-1}^{\max\{i,b\}-1}
	x_k^{\max\{nw_{k},c\}}x_{k+1}^{jw_{k+1}}x_{k+2}^{dw_{k+2}}.$$ 
	Note that $f'\in (x_{k-1}, x_{k+1}^{w_{k+1}})^{s-1}\cap (x_{k-2}, 
	x_{k}^{w_{k}})^{s-1}\cap (x_{k-3}, x_{k-1},x_k,x_{k+2}^{w_{k+2}})^{s-1}$ and 
	hence by induction hypothesis, $f'\in I^{s-1}$. Therefore we have $f\in I^s$. 
	If $m=0$, then we have $n=s$, and so $f=x_{k-3}^ax_{k-2}^mx_{k-1}^{\max\{i,b\}}
	x_k^{sw_{k}}x_{k+1}^{jw_{k+1}}x_{k+2}^{dw_{k+2}}$. Then divide $f$ by the 
	edge $x_{k-1}x_k^{sw_{k}}$ to get $$f'=x_{k-3}^ax_{k-2}^m
	x_{k-1}^{\max\{i,b\}-1}x_k^{(s-1)w_{k}}x_{k+1}^{jw_{k+1}}x_{k+2}^{dw_{k+2}}.$$ 
	Again note that $f'\in (x_{k-1}, x_{k+1}^{w_{k+1}})^{s-1}\cap 
	(x_{k-2}, x_{k}^{w_{k}})^{s-1}\cap (x_{k-3}, x_{k-1},x_k,x_{k+2}^{w_{k+2}})^{s-1}$ 
	and hence by induction hypothesis we are done. 
	If $\max\{i,b\}=0$, then $i=b=0$. 
	We have $j=s$ and $f=x_{k-3}^ax_{k-2}^mx_k^{\max\{nw_{k},c\}}x_{k+1}^{sw_{k+1}}
	x_{k+2}^{dw_{k+2}}$. If $n\neq 0$ and $x_k|f$, then divide the monomial 
	$f$ by the edge $x_kx_{k+1}^{w_{k+1}}$ to get 
	$$f'=x_{k-3}^ax_{k-2}^mx_k^{\max\{nw_{k},c\}-1}x_{k+1}^{(s-1)
	w_{k+1}}x_{k+2}^{dw_{k+2}}.$$ 
	Note that $f'\in (x_{k-1}, x_{k+1}^{w_{k+1}})^{s-1}\cap (x_{k-2}, 
	x_{k}^{w_{k}})^{s-1}\cap (x_{k-3}, x_{k-1},x_k,x_{k+2}^{w_{k+2}})^{s-1}$ and hence by 
	induction hypothesis we are done. If $n=0$, then $m=s$ and we have 
	$$f=x_{k-3}^ax_{k-2}^sx_k^{c}x_{k+1}^{sw_{k+1}}x_{k+2}^{dw_{k+2}}= 
	(x_{k-3}x_{k-2})^a(x_{k-2}x_k)^{c}(x_{k+1}x_{k+2}^{w_{k+2}})^d\cdot x_{k-2}^{s-(a+c)}
	x_{k+2}^{sw_{k+2}-d}\in I^s$$ as $a+c+d=s$. This completes the proof.
\end{proof}

\bibliographystyle{abbrv}
\bibliography{refs_reg} 
\end{document}